\documentclass{proc-l}
\newtheorem{theorem}{Theorem}[section]
\newtheorem{lemma}[theorem]{Lemma}

\theoremstyle{definition}

\newtheorem{proposition}{Proposition}
\theoremstyle{remark}

\usepackage{amssymb,colordvi}
\numberwithin{equation}{section}

\begin{document}

\title{New proofs of two $q$-analogues of Koshy's formula}
\author{Emma Yu Jin}
\address{Department of Computer Science, University of Kaiserslautern}
\curraddr{Kaiserslautern, Germany} \email{jin@cs.uni-kl.de}
\thanks{The work of the first author has been
supported by the research grants from DFG (Deutsche
Forschungsgemeinschaft), JI 207/1-1.}

\author{Markus E.\ Nebel}
\address{Department of Computer Science, University of Kaiserslautern}
\curraddr{Kaiserslautern, Germany}
\email{nebel@cs.uni-kl.de}
\thanks{}


\date{}

\dedicatory{}

\commby{Jim Haglund}

\begin{abstract}
In this paper we prove a $q$-analogue of Koshy's formula in terms of
the Narayana polynomial due to Lassalle and a $q$-analogue of
Koshy's formula in terms of $q$-hypergeometric series due to Andrews
by applying the inclusion-exclusion principle on Dyck paths and on
partitions. We generalize these two $q$-analogues of Koshy's formula
for $q$-Catalan numbers to that for $q$-Ballot numbers. This work
also answers an open question by Lassalle and two questions raised
by Andrews in 2010. We conjecture that if $n$ is odd, then for $m\ge
n\ge 1$, the polynomial $(1+q^n){m\brack n-1}_q$ is unimodal. If $n$
is even, for any even $j\ne 0$ and $m\ge n\ge 1$, the polynomial
$(1+q^n)[j]_q{m\brack n-1}_q$ is unimodal. This implies the answer
to the second problem posed by Andrews.

\end{abstract}

\maketitle

\section{Introduction and Background}
Let $C_n=\frac{1}{n+1}\binom{2n}{n}$ be the $n$-th Catalan number, a
recursive formula for Catalan numbers is given by Koshy as follows
\cite{Koshy}:
\begin{eqnarray}\label{E:Koshy1}
\sum_{r=0}^n(-1)^r\binom{n-r+1}{r}C_{n-r}=0.
\end{eqnarray}
Since this is a hypergeometric sum identity, it can be proved by
Zeilberger's creative telescoping method \cite{Zeil:91}.
Furthermore, the Koshy's formula (\ref{E:Koshy1}) follows
immediately from the orthogonality of a special Fibonacci polynomial
due to Cigler \cite{Cigler:96,Cigler:13}. However, these approaches
do not provide structural insight; to this end a combinatorial proof
based on the inclusion-exclusion principle is needed. Our proof of
(\ref{E:Koshy1}) in the context of labeled elevated Dyck paths is
exactly of this kind. Here, a Dyck path of length $2n$ is a
$2$-dimensional lattice path from $(0,0)$ to $(2n,0)$ with allowed
steps $(1,1)$ ($U$-step) and $(1,-1)$ ($D$-step) that never goes
below the $x$-axis. A Dyck path $p$ is an elevated Dyck path if
$p=U\bar{p}D$ where $\bar{p}$ is a Dyck path. We use $D_n$ to denote
the set of elevated Dyck paths of length $2n+2$. Then $\vert
D_n\vert=C_n$. We denote by $P$ the peak $UD$. We call $h$ an
up-peak if $h=UUD$. Let $\mathcal{UP}(p)$ (resp.\ $\mathcal{U}(p)$)
be the set of up-peaks (resp.\ $U$-steps) contained in a path $p$.
We denote by $(\mathcal{M}^s,j,D_n)$ the set of elevated Dyck paths
$p\in D_n$ having exactly $j$ elements from the set $\mathcal{M}(p)$
labeled by $s$. We set
$a_{n,j}(\mathcal{M})=\vert(\mathcal{M}^s,j,D_n)\vert$. For the case
$\mathcal{M}=\mathcal{UP}$, the number $a_{n,j}(\mathcal{UP})$ can
be counted in two different ways.

First, there is a bijection between the set
$((\mathcal{UP})^s,j,D_n)$ and the set $(\mathcal{U}^s,j,D_{n-j})$.
For a given path $p\in D_n$ that has $j$ up-peaks labeled by $s$, we
can get a unique path $p'\in D_{n-j}$ that has $j$ $U$-steps labeled
by $s$ if we remove the peak $P$ from each labeled up-peak. As a
result, we have
\begin{align}\label{E:agm}
a_{n,j}(\mathcal{UP})=\left\vert((\mathcal{UP})^s,j,D_n)\right\vert
=\left\vert(\mathcal{U}^s,j,D_{n-j})\right\vert=\binom{n-j+1}{j}C_{n-j}.
\end{align}
Second, we set $g_{n,m}=\vert\{p\in D_n: \vert
\mathcal{UP}(p)\vert=m\}\vert$. For a given elevated Dyck path $p\in
D_n$ such that $\vert \mathcal{UP}(p)\vert=m$, there are
$\binom{m}{j}$ ways to label $j$ up-peaks by $s$, which leads to
$a_{n,j}(\mathcal{UP})=\sum_{m\ge j}\binom{m}{j}g_{n,m}$. In
combination of (\ref{E:agm}), we obtain
\begin{align}\label{E:iekoshy}
\sum_{m\ge j}{m\choose
j}g_{n,m}=a_{n,j}(\mathcal{UP})=\binom{n-j+1}{j}C_{n-j}.
\end{align}
Let $G_n(q)$ be the ordinary generating function of $g_{n,m}$, i.e., $
\sum_{m\ge 1}g_{n,m}q^m=G_n(q)$. Then from (\ref{E:iekoshy}), we
can derive
\begin{eqnarray}\label{E:1}
G_n(q)=\sum_{j\ge 0}\frac{G_n^{(j)}(1)}{j!}(q-1)^j=\sum_{j\ge
0}\binom{n-j+1}{j}C_{n-j}(q-1)^j,
\end{eqnarray}
where $G_n^{(j)}(1)=\frac{\partial^j G_n(q)}{\partial
q^j}\vert_{q=1}$. By setting $q=0$ on both sides of (\ref{E:1}), the
Koshy's formula (\ref{E:Koshy1}) follows. Our proof uses the
generating function approach to illustrate the inclusion-exclusion
principle, see Chapter 2.3 of \cite{Stanley:ec2} and Chapter 4.2 of
\cite{Wilfbook}. Here we will use this approach to prove
(\ref{E:koshyLa}) and (\ref{E:cnq1in1}). We can also prove the
sieving identity by a direct application of the inclusion-exclusion
method. To make this point clear, we will prove (\ref{E:exin2}) in
both ways in Section \ref{S:prove1}. Next we give two $q$-analogues
of Koshy's formula due to Lassalle and Andrews
\cite{Andrews:08,Lassalle:11}. The Narayana number $N_{n,k}$ and the
Narayana polynomial $N_n(q)$ are defined by
\begin{equation}
N_{n,k}=\frac{1}{n}{n\choose k-1}{n\choose k},\quad
N_n(q)=\sum_{k=1}^n N_{n,k}q^{k-1},
\end{equation}
where $N_{n,k}$ counts the number of Dyck paths of length $2n$ that
have $k$ peaks $P$. By utilizing a $\lambda$-identity of complete
functions, Lassalle \cite{Lassalle:11} proved for $n\ge 1$,
\begin{align}\label{E:koshyLa}
N_n(q)&=(1-q)^{n-1}+q\sum_{k=1}^{n-1}N_{n-k}(q)\sum_{m=0}^{k-1}
(-1)^m\binom{k-1}{m}\binom{n-m}{k}(1-q)^{k-m-1}.
\end{align}
Furthermore, Andrews gave another $q$-analogue of Koshy's formula in
terms of $q$-hypergeometric series. Here we adopt the standard
notations of $q$-series, i.e.,
\begin{align*}
[n]_q&=1+q+\cdots+q^{n-1},\,\, [n]_q!=[1]_q[2]_q\cdots[n]_q,\\
(x;q)_n&=(1-x)(1-qx)\cdots(1-q^{n-1}x),\\
{m\brack n}_q&=\frac{[m]_q!}{[n]_q![m-n]_q!}
=\frac{(q;q)_m}{(q;q)_n(q;q)_{m-n}}, \\
C_n(q)&=\frac{1}{[n+1]_q}{2n \brack n}_q.
\end{align*}
Then before mentioned $q$-analogue of Koshy's formula is given by
\begin{eqnarray}
C_n(q)&=&\sum_{r=1}^n(-1)^{r-1}T_r(n,q),\label{E:cnq1in1}\\
T_r(n,q)&=&q^{r^2-r}\frac{(-q^{n-r+1};q)_r}{(-q;q)_r}\cdot
{n-r+1\brack r}_q\cdot C_{n-r}(q).\label{E:cnq1in}
\end{eqnarray}
Andrews \cite{Andrews:08} raised open questions about $T_r(n,q)$,
which are:
\begin{enumerate}
  \item Is $T_r(n,q)$ a polynomial in $q$?
  \item Does $T_r(n,q)$ only have nonnegative coefficients of $q$ if $n\ge 2r$?
  \item Does $T_{r-1}(2r-1,-q)$ only have nonnegative coefficients?
  \item What is the partition-theoretic combinatorial interpretation of $T_r(n,q)$ for $n\ge 2r$,
        what is that of $T_{r-1}(2r-1,-q)$, and what is the sieving process
        on the partitions to eliminate all the non-Catalan partitions?
\end{enumerate}
In Section~\ref{S:Tr} we completely answer (1) and (3), and show (2)
for even $n$. Furthermore, in Sections~\ref{S:prove1}
and~\ref{S:prove2} we prove (\ref{E:koshyLa}) and
(\ref{E:cnq1in1}) by the inclusion-exclusion method. In
Sections~\ref{S:ball} and \ref{S:qball} we generalize
(\ref{E:koshyLa}) and (\ref{E:cnq1in1}) for
$N_n(q)$ and $C_n(q)$ to that for Ballot numbers $B_{n,r}$ and
$q$-Ballot numbers $B_r(n,q)$ . Finally we conjecture that if $n$ is
odd, then for $m\ge n\ge 1$, the polynomial $(1+q^{n}){m\brack
n-1}_q$ is unimodal. If $n$ is even, for any even $j\ne 0$ and $m\ge
n\ge 1$, the polynomial $(1+q^n)[j]_q{m\brack n-1}_q$ is unimodal.
This would answer question (2) from above for odd $n$.

\section{Proof of (\ref{E:koshyLa})}\label{S:prove1}
Before we prove (\ref{E:koshyLa}), we give some main ingredients
needed in the proof.

First we introduce the notions of a {\it tower} and a {\it colored
tower} for an elevated Dyck path. We will choose to label the
colored towers by $s$ in order to apply the inclusion-exclusion
method. Recall that $D_n$ is the set of elevated Dyck paths of
length $2n+2$. For any $p\in D_n$, we will use $U^iD^j\subseteq p$
to express the fact that $p$ contains $i$ consecutive $U$-steps that
are followed by $j$ consecutive $D$-steps. We call $t$ {\it a tower}
of height $i$ contained in an elevated Dyck path $p$, if
$t=U^iD^i\subseteq p$ and $UtD\not\subseteq p$. For any elevated
Dyck path $p\in D_n$, let $\mathcal{T}^*(p)$ be the set of towers
contained in $p$. We will color the towers in $\mathcal{T}^*(p)$ in
the following way:
For any $p\in D_n$,\\
Step $1$: color all the towers $t\subseteq p$ that immediately
follow a $U$-step;\\
Step $2$: color all the towers $t\subseteq p$
that immediately follow an uncolored tower $t'\subseteq p$ after Step $1$.\\[1mm]
We use $\mathcal{T}(p)\subseteq \mathcal{T}^*(p)$ to denote the set
of towers contained in a path $p$ which are colored according to
Step 1 and 2. We call every $t$ in the set $\mathcal{T}(p)$ {\it a
colored tower}. We set
$\mathcal{T}^c(p)=\mathcal{T}^*(p)-\mathcal{T}(p)$ and let
$\mathcal{T}_2(p)\subseteq \mathcal{T}(p)$ (resp.\
$\mathcal{T}_1(p)\subseteq \mathcal{T}(p)$) be the set of colored
towers of height $\ge 2$ (resp.\ height $=1$). As an example,
consider the path $p=U^3D^2UDUD^2\in D_4$. The colored towers in the
set $\mathcal{T}(p)$ are depicted using double lines in the
rightmost elevated Dyck path of the figure below. After Step~$1$,
the first tower in the left-to-right order is colored. The second
tower remains uncolored after Step $1$, and therefore after Step
$2$, the third tower is colored.
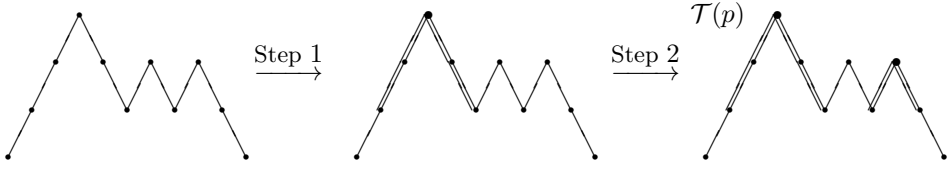
\begin{center}
\setlength{\unitlength}{3pt}
\begin{picture}(8,23)(40,6)
\put(-16,8){\circle*{0.5}}
\put(-16,8){\line(1,2){3}}\put(-13,14){\circle*{0.5}}
\put(-13,14){\line(1,2){3}}\put(-10,20){\circle*{0.5}}
\put(-10,20){\line(1,2){3}}\put(-7,26){\circle*{0.5}}
\put(-7,26){\line(1,-2){3}}\put(-4,20){\circle*{0.5}}
\put(-4,20){\line(1,-2){3}}\put(-1,14){\circle*{0.5}}
\put(-1,14){\line(1,2){3}}\put(2,20){\circle*{0.5}}
\put(2,20){\line(1,-2){3}}\put(5,14){\circle*{0.5}}
\put(5,14){\line(1,2){3}}\put(8,20){\circle*{0.5}}
\put(8,20){\line(1,-2){3}}\put(11,14){\circle*{0.5}}
\put(11,14){\line(1,-2){3}}\put(14,8){\circle*{0.5}}
\put(15,20){$\underrightarrow{\small{\mbox{Step }}1}$}
\put(28,8){\circle*{0.5}}
\put(28,8){\line(1,2){3}}\put(31,14){\circle*{0.5}}
\put(30.5,14){\line(1,2){3}}
\put(31,14){\line(1,2){3}}\put(34,20){\circle*{0.5}}
\put(33.5,20){\line(1,2){3}}
\put(34,20){\line(1,2){3}}\put(37,26){\circle*{1}}
\put(36.5,26){\line(1,-2){3}}
\put(37,26){\line(1,-2){3}}\put(40,20){\circle*{0.5}}
\put(39.5,20){\line(1,-2){3}}
\put(40,20){\line(1,-2){3}}\put(43,14){\circle*{0.5}}
\put(43,14){\line(1,2){3}}\put(46,20){\circle*{0.5}}
\put(46,20){\line(1,-2){3}}\put(49,14){\circle*{0.5}}
\put(49,14){\line(1,2){3}}\put(52,20){\circle*{0.5}}
\put(52,20){\line(1,-2){3}}\put(55,14){\circle*{0.5}}
\put(55,14){\line(1,-2){3}}\put(58,8){\circle*{0.5}}
\put(60,20){$\underrightarrow{\small{\mbox{Step }}2}$}
\put(70,25){$\mathcal{T}(p)$} \put(72,8){\circle*{0.5}}
\put(72,8){\line(1,2){3}}\put(75,14){\circle*{0.5}}
\put(74.5,14){\line(1,2){3}}
\put(75,14){\line(1,2){3}}\put(78,20){\circle*{0.5}}
\put(77.5,20){\line(1,2){3}}
\put(78,20){\line(1,2){3}}\put(81,26){\circle*{1}}
\put(80.5,26){\line(1,-2){3}}
\put(81,26){\line(1,-2){3}}\put(84,20){\circle*{0.5}}
\put(83.5,20){\line(1,-2){3}}
\put(84,20){\line(1,-2){3}}\put(87,14){\circle*{0.5}}
\put(87,14){\line(1,2){3}}\put(90,20){\circle*{0.5}}
\put(90,20){\line(1,-2){3}}\put(93,14){\circle*{0.5}}
\put(92.5,14){\line(1,2){3}}
\put(93,14){\line(1,2){3}}\put(96,20){\circle*{1}}
\put(95.5,20){\line(1,-2){3}}
\put(96,20){\line(1,-2){3}}\put(99,14){\circle*{0.5}}
\put(99,14){\line(1,-2){3}}\put(102,8){\circle*{0.5}}
\end{picture}
\end{center}
Second we will use some one-to-one correspondences between two sets
of labeled elevated Dyck paths. If there is a bijection between two
sets $A$ and $B$, we write $A\simeq B$. Recall that
$(\mathcal{M}^s,j,D_n)$ is the set of elevated Dyck paths $p\in D_n$
having exactly $j$ elements from the set $\mathcal{M}(p)$ labeled by
$s$. We consider the set of elevated Dyck paths $p\in D_n$ that have
$m$ colored towers labeled by $s$, which is the particular case
$\mathcal{M}=\mathcal{T}$ and $j=m$. Since each colored tower has
height $1$ or at least $2$, we have
$(\mathcal{T}^s,m,D_n)=((\mathcal{T}_1\cup\mathcal{T}_2)^s,m,D_{n})$.
Let $\mathcal{U}^2(p)$ be the set of $UU$-steps of an elevated Dyck
path $p\in D_n$, we will show
\begin{lemma}\label{L:1}
$(\mathcal{T}_1^s,1,D_n)\simeq ((\mathcal{U}^2\cup \mathcal{T}^c)^s,1,D_{n-1})$, $(\mathcal{T}_2^s,1,D_n)\simeq (\mathcal{T}^s,1,D_{n-1})$ and
$(\mathcal{T}^s,m,D_n)\simeq (\mathcal{U}^s,m,D_{n-m})$ for any $m$, $1\le m\le n$.
\end{lemma}
\begin{proof}
We use $U^s,U^sU^s,t^s$ to represent a $U$-step, a $UU$-step and a
tower labeled by $s$. For a path $p\in D_n$ that has only one
colored tower $t$ labeled by $s$, if the labeled and colored tower
$t$ has height $1$, i.e., $t\in\mathcal{T}_1(p)$, then $t$ is either
located between two $U$-steps, i.e., $Ut^sU\subseteq p$, or $t$
follows an uncolored tower $t_1$ from $\mathcal{T}^c(p)$, i.e.,
$t_1t^s\subseteq p$ where $t_1\in\mathcal{T}^c(p)$. Let
$g_1:(\mathcal{T}_1^s,1,D_n)\rightarrow ((\mathcal{U}^2\cup
\mathcal{T}^c)^s,1,D_{n-1})$ be the map defined as follows. If
$p=\cdots Ut^sU\cdots \in(\mathcal{T}_1^s,1,D_n)$, then
$g_1(p)=\cdots U^sU^s\cdots\in ((\mathcal{U}^2)^s,1,D_{n-1})$. If
$p=\cdots t_1t^s\cdots\in (\mathcal{T}_1^s,1,D_n)$, then
$g_1(p)=\cdots t_1^s\cdots\in ((\mathcal{T}^c)^s,1,D_{n-1})$. Notice
that the tower $t_1$ is also an uncolored tower of the labeled path
$g_1(p)$. This follows from the way we define the colored towers. It
is obvious that the map $g_1$ is a bijection.

For a path $p\in D_n$ that has only one colored tower $t$ labeled by
$s$, if the labeled and colored tower $t$ has height at least $2$,
i.e., $t\in\mathcal{T}_2(p)$, then $t=Ut_2D$ where $t_2$ is a tower.
Let $g_2:(\mathcal{T}_2^s,1,D_n)\rightarrow
(\mathcal{T}^s,1,D_{n-1})$ be the map defined as follows. If
$p=\cdots t^s\cdots=\cdots
(Ut_2D)^s\cdots\in(\mathcal{T}_2^s,1,D_n)$, then $g_2(p)=\cdots
t_2^s\cdots\in(\mathcal{T}^s,1,D_{n-1})$. Notice that the tower
$t_2$ is also a colored tower of the labeled path $g_2(p)$. This
follows from the way we define the colored towers. It is clear that
the map $g_2$ is a bijection.

We will next prove $(\mathcal{T}^s,m,D_n)\simeq
((\mathcal{U}^2\cup\mathcal{T} \cup\mathcal{T}^c)^s,m,D_{n-m})$ for
any $m$. The bijection for the case $m=1$ can be obtained by
combining the bijections $g_1$ and $g_2$, namely, let
$f_1:(\mathcal{T}^s,1,D_n)\rightarrow ((\mathcal{U}^2\cup
\mathcal{T}^c\cup\mathcal{T})^s,1,D_{n-1})$ be the bijection defined
as follows. $f_1(p)=g_1(p)$ if $p\in(\mathcal{T}_1^s,1,D_n)$ and
$f_1(p)=g_2(p)$ if $p\in(\mathcal{T}_2^s,1,D_n)$. We can extend the
bijection $f_1$ to the general bijection $f_m$ from the set
$(\mathcal{T}^s,m,D_n)$ to the set
$((\mathcal{U}^2\cup\mathcal{T}\cup\mathcal{T}^c)^s,m,D_{n-m})$ as
follows. For any $p\in (\mathcal{T}^s,m,D_n)$, suppose $p$ has
$t_1,t_2,\ldots,t_m$ colored towers labeled by $s$ from left to
right, for every labeled tower $t_i$ that has height $1$ and
$Ut_i^sU\subseteq p$, we label the $U$-steps next to $t_i$ by $s$
and remove $t_i$ from $p$, i.e., $U^sU^s\subseteq f_m(p)$. For every
labeled tower $t_i$ that has height $1$ and $t\,t_i^s\subseteq p$,
we label the tower $t$ by $s$ and remove $t_i$ from $p$, i.e.,
$t^s\subseteq f_m(p)$. For every labeled tower $t_i$ that has height
at least $2$, i.e., $t_i^s=(UtD)^s\subseteq p$, we label $t$ by $s$
and remove a $U$-step and a $D$-step from the tower $t_i$, i.e.,
$t^s\subseteq f_m(p)$. In fact the map $f_m$ on the path $p$ is
equivalent to applying the bijection $f_1$ on each labeled tower
$t_i$ of $p$ from left to right. It follows that $f_m$ is a
bijection and therefore $(\mathcal{T}^s,m,D_n)\simeq
((\mathcal{U}^2\cup\mathcal{T} \cup\mathcal{T}^c)^s,m,D_{n-m})$
holds for any $m$. It remains to prove
$\mathcal{U}^2(p)\cup\mathcal{T}(p) \cup\mathcal{T}^c(p)\simeq
\mathcal{U}(p)$ for any $p\in D_n$.

In view of
$\mathcal{T}^c(p)\cup \mathcal{T}(p)=\mathcal{T}^*(p)$ for any $p\in D_n$, we
next show for any $p\in D_n$, $\mathcal{U}^2(p)\cup
\mathcal{T}^*(p)\simeq \mathcal{U}(p)$. Each $U$-step is either part
of a $UU$-step or part of a peak $P$ by considering the step that
follows this $U$-step. Each peak $P$ is contained in only one tower
from $\mathcal{T}^*(p)$ and each tower from $\mathcal{T}^*(p)$ has
only one peak. Therefore $\mathcal{U}(p)\simeq \mathcal{U}^2(p)\cup
\mathcal{T}^*(p)$ for any $p\in D_n$ and the proof is complete.
\end{proof}
For a path $p$, we name the first $U$-step of a tower
$t\in\mathcal{T}_2(p)$ {\it bottom}. Let $\mathcal{B}(p)$ be the set
of bottoms contained in a path $p$, let
$((\mathcal{T}^s,\mathcal{B}^w),(m,r),D_n)$ be the set of elevated
Dyck paths $p\in D_n$ having exactly $m$ colored towers labeled by
$s$ and among these $m$ colored towers there are exactly $r$ bottoms
labeled by $w$. Let $((\mathcal{T}^s,\mathcal{T}^w),(m,r),D_n)$ be
the set of elevated Dyck paths $p\in D_n$ having exactly $m$ colored
towers labeled by $s$ and among these $m$ colored towers there are
exactly $r$ colored towers labeled by $w$. We will prove
\begin{lemma}\label{L:2}
$((\mathcal{T}^s,\mathcal{B}^w),(m,r),D_n)\simeq ((\mathcal{T}^s,\mathcal{T}^w),(m,r),D_{n-r})$.
\end{lemma}
\begin{proof}
For a path $p\in ((\mathcal{T}^s,\mathcal{B}^w),(m,r),D_n)$ and for
every colored tower $t\in\mathcal{T}_2(p)$ labeled by $s$ whose
bottom is labeled by $w$, we remove the bottom and a $D$-step from
$t$, and label the remaining colored tower by $s$ and $w$. For a
path $\bar{p}\in ((\mathcal{T}^s,\mathcal{T}^w),(m,r),D_{n-r})$ and
for every colored tower $\bar{t}\in \mathcal{T}(\bar{p})$ labeled by
$s,w$, we replace the tower $\bar{t}^{s,w}$ by the tower
$U^w\bar{t}^sD$. This yields a bijection.
\end{proof}
Now we are in position to prove (\ref{E:koshyLa}).
\begin{proof}
Let $F_n(q)=qN_n(q)$, then (\ref{E:koshyLa}) is equivalent to
\begin{align}\label{E:exin}
\frac{F_n(q)}{(1-q)^{n}}=\frac{q}{1-q}+\sum_{k=1}^{n-1}
\frac{F_{n-k}(q)}{(1-q)^{n-k}}\sum_{m=0}^{k-1}(-1)^m\binom{k-1}{m}
\binom{n-m}{k}\frac{q}{(1-q)^{m+1}},
\end{align}
where $\frac{F_n(q)}{(1-q)^n}=\sum_{k=1}^n
\frac{N_{n,k}q^k}{(1-q)^n}$ counts the number of elevated Dyck paths of
length $2n+2$ where each $U$-step, except the first $U$-step, has weight $\frac{1}{1-q}$ and each
peak $UD$ has weight $q$. Before we proceed to the combinatorial
proof, we first transform $\frac{q}{(1-q)^{m+1}}$ according to our
needs, namely,
\begin{align*}
&\quad\sum_{m=0}^{k-1}(-1)^m\binom{k-1}{m}\binom{n-m}{k}\frac{q}{(1-q)^{m+1}}\\
&=\sum_{i=0}^{k-1}\sum_{m\ge
i}(-1)^m\binom{k-1}{m}\binom{n-m}{k}\binom{m}{i}(\frac{q}{1-q})^{i+1}\\
&=\sum_{m=1}^{k}(-1)^{m-1}\binom{k-1}{m-1}\binom{n-k+1}{m}(\frac{q}{1-q})^{m}.
\end{align*}
Consequently, we can express (\ref{E:exin}) as
\begin{align}\label{E:exin3}
\frac{F_n(q)}{(1-q)^{n}}=\frac{q}{1-q}+\sum_{m=1}^{n-1}(-1)^{m-1}
\sum_{k=m}^{n-m+1}\frac{F_{n-k}(q)}{(1-q)^{n-k}}
\binom{k-1}{m-1}\binom{n-k+1}{m}(\frac{q}{1-q})^{m}.
\end{align}
We observe that adding a peak $P$ right between double $U$ steps
contributes weight $\frac{q}{1-q}$, while adding a peak $P$ right
between $UD$-steps contributes weight $\frac{1}{1-q}$. In contrast
to the proof of (\ref{E:Koshy1}), the reduction from
$\frac{F_n(q)}{(1-q)^n}$ to $\frac{F_{n-1}(q)}{(1-q)^{n-1}}$ by
removing a peak highly depends on the type of the $U$-step we
choose.

Next we will derive an identity for the generating function of
weighted elevated Dyck paths. This can also be obtained by directly
applying the inclusion-exclusion principle. We will show this identity
in both ways. We set $g_{n,k}(\mathcal{T})=\vert\{p\in D_n:
\vert\mathcal{T}(p)\vert=k\}\vert$ and define
$g_{n,k}(\mathcal{T},q)$ to be the generating function of
$g_{n,k}(\mathcal{T})$ that weights each $U$-step, except the first
$U$-step, of an elevated Dyck path $p\in D_n$ by $\frac{1}{1-q}$ and
each peak $UD$ of an elevated Dyck path $p\in D_n$ by $q$. Recall
that we denote by $(\mathcal{M}^s,m,D_n)$ the set of elevated Dyck
paths $p\in D_n$ having exactly $m$ elements from the set
$\mathcal{M}(p)$ labeled by $s$, and that
$a_{n,m}(\mathcal{M})=\vert (\mathcal{M}^s,m,D_n)\vert$.
Furthermore, we denote by $((\cup_{i=1}^k\mathcal{M}_i)^s,m,D_n)$
the set of elevated Dyck paths $p\in D_n$ having exactly $m$
elements from the set $\cup_{i=1}^k\mathcal{M}_i(p)$ labeled by $s$,
and define $a_{n,m}(\cup_{i=1}^k\mathcal{M}_i)=\vert
((\cup_{i=1}^k\mathcal{M}_i)^s,m,D_n)\vert$. Then, analogous to the
proof of (\ref{E:Koshy1}), we consider
$\mathcal{M}=\mathcal{T}$. Let $a_{n,m}(\mathcal{T},q)$ be the
generating function of $a_{n,m}(\mathcal{T})$ where each $U$-step,
except the first $U$-step, of an elevated Dyck path $p\in D_n$ is
weighted by $\frac{1}{1-q}$ and each peak $P$ of an elevated Dyck
path $p\in D_n$ is weighted by $q$. Then
\begin{align*}
a_{n,m}(\mathcal{T})=\sum_{k\ge
m}\binom{k}{m}g_{n,k}(\mathcal{T}),\quad\,
a_{n,m}(\mathcal{T},q)=\sum_{k\ge
m}\binom{k}{m}g_{n,k}(\mathcal{T},q).
\end{align*}
In particular for $m=0$, we have $a_{n,0}(\mathcal{T})=C_n$ and
$a_{n,0}(\mathcal{T},q)=\frac{F_n(q)}{(1-q)^n}$. By introducing the
generating function $G_n(x,q)$ of $g_{n,k}(\mathcal{T},q)$, we
obtain
\begin{align*}
G_n(x,q)=\sum_{k\ge 1}g_{n,k}(\mathcal{T},q)x^k=\sum_{m=0}^{n}
\frac{G_n^{(m)}(1,q)}{m!}(x-1)^m=\sum_{m=0
}^{n}a_{n,m}(\mathcal{T},q)(x-1)^m
\end{align*}
where $G_n^{(m)}(1,q)=\frac{\partial^m\,G_n(x,q)}{\partial
x^m}|_{x=1}$. By setting $x=0$, we get
\begin{eqnarray}\label{E:exin2}
a_{n,0}(\mathcal{T},q)=\frac{F_n(q)}{(1-q)^n}=\sum_{m=1}^n(-1)
^{m-1}a_{n,m}(\mathcal{T},q).
\end{eqnarray}
In fact (\ref{E:exin2}) can also be derived by directly applying the
inclusion-exclusion method. To be precise, we adopt the notations
from \cite{Stanley:ec2}. Let $A_i$ be the set of elevated Dyck paths
in $D_n$ whose $i$-th tower (in the left-to-right order) is a
colored tower. Then $A_1,\ldots,A_n$ are all the subsets of the set
$D_n$ and $A_1\cup\cdots \cup A_n=D_n$. For each subset $T$ of
$[n]$, let $A_T=\bigcap_{i\in T}A_i$ with $A_{\varnothing}=D_n$.
Here $A_T$ is the set of elevated Dyck paths in $D_n$ whose $i$-th
tower (in the left-to-right order) is a colored tower for every
$i\in T$. For $0\le m\le n$, we set $S_m=\sum_{\vert T\vert=m}\vert
A_T\vert$ where the sum runs over all the $m$-subsets $T$ of $[n]$.
Therefore $S_m$ equals to the number of elevated Dyck paths in $D_n$
that have $m$ colored towers labeled by $s$, which is
$a_{n,m}(\mathcal{T})$. So according to the principle of
inclusion-exclusion, the number $\vert \bar{A}_1\cap\cdots\cap
\bar{A}_n\vert$ of elevated Dyck paths in $D_n$ that have no colored
towers equals
\begin{align*}
\vert \bar{A}_1\cap\cdots\cap \bar{A}_n\vert=\sum_{m=0}^{n}(-1)^mS_m
=\sum_{m=0}^n(-1)^ma_{n,m}(\mathcal{T}).
\end{align*}
Since every elevated Dyck path has at least one colored tower, we
have $\vert \bar{A}_1\cap\cdots\cap
\bar{A}_n\vert=g_{n,0}(\mathcal{T})=0$. Consequently,
$a_{n,0}(\mathcal{T})=\sum_{m=1}^n(-1) ^{m-1}a_{n,m}(\mathcal{T})$.
In terms of the weighted elevated Dyck paths, (\ref{E:exin2})
follows.

In order to use (\ref{E:exin2}) to prove (\ref{E:exin3}),
we will need to give an explicit
expression for $a_{n,m}(\mathcal{T},q)$. For any path $p\in D_n$
and any two colored towers $t_1,t_2\in\mathcal{T}(p)$, $t_1$ and
$t_2$ are disjoint. We start by
counting $a_{n,1}(\mathcal{T},q)$, which is the generating function for the elevated
Dyck paths $p\in D_n$ having one colored tower labeled by $s$. Equivalently,
$\vert(\mathcal{T}^s,1,D_n)\vert=a_{n,1}(\mathcal{T})$.
In view of
Lemma~\ref{L:1}, we get $a_{n,1}(\mathcal{T})=a_{n-1,1}(\mathcal{U})$ and
\begin{align}
\nonumber
a_{n,1}(\mathcal{T},q)&=\frac{q}{1-q}a_{n-1,1}(\mathcal{U}^2\cup\mathcal{T}^c,q)
+\frac{1}{1-q}a_{n-1,1}(\mathcal{T},q)\\
\nonumber
&=\frac{q}{1-q}a_{n-1,1}(\mathcal{U},q)+a_{n-1,1}(\mathcal{T},q)
=\sum_{k=1}^{n-1}\frac{q}{1-q}a_{n-k,1}(\mathcal{U},q)+a_{1,1}(\mathcal{T},q)\\
\label{E:an11}&=\sum_{k=1}^{n-1}(n-k+1)\frac{F_{n-k}(q)}{(1-q)^{n-k}}\frac{q}{1-q}
+\frac{q}{1-q}.
\end{align}
Next we will count $a_{n,m}(\mathcal{T},q)$ for $m\ge 2$. For a path
$p$, recall that the first $U$-step of a tower
$t\in\mathcal{T}_2(p)$ is named bottom. For a given path $p\in D_n$
with $m$ colored towers labeled by $s$, let $\mathcal{T}_{2,m}$
denote the subset of these $m$ towers that have height $\geq 2$.
Since each $U$-step has weight $\frac{1}{1-q}$ and each tower in
$\mathcal{T}_{2,m}$ has one bottom, the weight on the bottoms of
path $p\in D_n$ is $(\frac{1}{1-q})^{\vert \mathcal{T}_{2,m}\vert}$,
which is equal to
\begin{eqnarray}\label{E:m2}
(\frac{1}{1-q})^{\vert
\mathcal{T}_{2,m}\vert}=(\frac{q}{1-q})^{\vert
\mathcal{T}_{2,m}\vert} +\sum_{r=1}^{\vert
\mathcal{T}_{2,m}\vert}(-1)^{r-1}\binom{\vert
\mathcal{T}_{2,m}\vert}{r}(\frac{1}{1-q})^{\vert
\mathcal{T}_{2,m}\vert-r}.
\end{eqnarray}
We will separate the weight $(\frac{1}{1-q})^{\vert
\mathcal{T}_{2,m}\vert}$ on the bottoms of path $p$ according to
(\ref{E:m2}). The term $(\frac{q}{1-q})^{\vert
\mathcal{T}_{2,m}\vert}$ on the right-hand side of (\ref{E:m2})
corresponds to the case that each bottom is weighted by
$\frac{q}{1-q}$. Accordingly, we set $f_{n,m,0}(\mathcal{T},q)$ to
be the generating function of elevated Dyck paths $p\in D_n$ that
have $m$ colored towers labeled by $s$ where each bottom has weight
$\frac{q}{1-q}$, each $U$-step, other than the first $U$-step and
the bottom, has weight $\frac{1}{1-q}$ and each peak has weight $q$.
In the same way as for the counting of $a_{n,1}(\mathcal{T},q)$, the
set $(\mathcal{T}^s,m,D_n)$ is in one-to-one correspondence to the
set $(\mathcal{U}^s,m,D_{n-m})$ as shown in Lemma~\ref{L:1}. Therefore,
\begin{align}\label{E:fm0}
a_{n,m}(\mathcal{T})&=a_{n-m,m}(\mathcal{U}),\nonumber\\
f_{n,m,0}(\mathcal{T},q)
&=(\frac{q}{1-q})^ma_{n-m,m}(\mathcal{U},q).
\end{align}
The term $\binom{\vert
\mathcal{T}_{2,m}\vert}{r}(\frac{1}{1-q})^{\vert
\mathcal{T}_{2,m}\vert-r}$ on the right-hand side of (\ref{E:m2})
represents the number of ways to choose $r$ towers from the set
$\mathcal{T}_{2,m}$ where each bottom of these $r$ towers has weight
$1$ and each bottom of the remaining
$(\vert\mathcal{T}_{2,m}\vert-r)$ towers has weight $\frac{1}{1-q}$.
Accordingly, for $r\ge 1$ we define $f_{n,m,r}(\mathcal{T})$ to be
the number of elevated Dyck paths $p\in D_n$ that have $m$ colored
towers labeled by $s$ and among these $m$ colored towers there are
exactly $r$ bottoms labeled by $w$, i.e.,
$\vert((\mathcal{T}^s,\mathcal{B}^w),(m,r),D_n)\vert=f_{n,m,r}(\mathcal{T})$.
Furthermore, let $f_{n,m,r}(\mathcal{T},q)$ be the generating
function of $f_{n,m,r}(\mathcal{T})$ where each bottom labeled by
$w$ has weight $1$, each $U$-step, other than the bottom labeled by
$w$ and the first $U$-step, has weight $\frac{1}{1-q}$ and each peak
has weight $q$. From Lemma~\ref{L:2}, we obtain
$f_{n,m,r}(\mathcal{T})=\vert((\mathcal{T}^s,\mathcal{T}^w),(m,r),D_{n-r})\vert$
and therefore
\begin{align}\label{E:fmr}
f_{n,m,r}(\mathcal{T})=\binom{m}{r}a_{n-r,m}(\mathcal{T}),\quad
f_{n,m,r}(\mathcal{T},q)=\binom{m}{r}a_{n-r,m}(\mathcal{T},q).
\end{align}
By multiplying both sides of (\ref{E:m2}) by the weights on the
$U$-steps (except the first $U$-step and the bottoms) and the
weights on the peaks of a path $p\in D_n$, and summing over all the
paths in $D_n$, we get
\begin{align*}
a_{n,m}(\mathcal{T},q)&=f_{n,m,0}(\mathcal{T},q)
+\sum_{r=1}^{m}(-1)^{r-1}f_{n,m,r}(\mathcal{T},q).
\end{align*}
Together with (\ref{E:fm0}) and (\ref{E:fmr}), we have
\begin{align}\label{E:recLas}
\nonumber
a_{n,m}(\mathcal{T},q)&=(\frac{q}{1-q})^ma_{n-m,m}(\mathcal{U},q)
+\sum_{r=1}^{m}(-1)^{r-1}\binom{m}{r}a_{n-r,m}(\mathcal{T},q),\\
&=(\frac{q}{1-q})^m\binom{n-m+1}{m} \frac{F_{n-m}(q)}{(1-q)^{n-m}}
+\sum_{r=1}^{m}(-1)^{r-1}\binom{m}{r}a_{n-r,m}(\mathcal{T},q).
\end{align}
We will employ generating functions to solve (\ref{E:recLas}).
Let $A_m(x)$ be the generating function for $a_{n,m}(\mathcal{T},q)$
and $m\ge 2$, i.e., the $n$-th coefficient of $A_m(x)$ -- denoted by
$[x^n]A_m(x)$ -- is $a_{n,m}(\mathcal{T},q)$. By multiplying both sides of
(\ref{E:recLas}) by $x^n$ and summing over all $n$, we obtain
\begin{align*}
A_m(x)&=(1-x)^{-m}\sum_{n\ge
m}\binom{n-m+1}{m}(\frac{q}{1-q})^m\frac{F_{n-m}(q)}{(1-q)^{n-m}}x^n,\\
a_{n,m}(\mathcal{T},q)&=[x^n]A_m(x)=\sum_{i\ge
0}\binom{n-m-i+1}{m}(\frac{q}{1-q})^m\binom{m+i-1}{m-1}
\frac{F_{n-m-i}(q)}{(1-q)^{n-m-i}}\\
&=\sum_{k=m}^{n-m+1}\binom{n-k+1}{m}\binom{k-1}{m-1}
\frac{F_{n-k}(q)}{(1-q)^{n-k}}(\frac{q}{1-q})^m.
\end{align*}
In combination of (\ref{E:exin2}) and (\ref{E:an11}), the
proof of (\ref{E:exin3}) is complete.
\end{proof}
\section{Generalize (\ref{E:koshyLa}) to Ballot numbers}\label{S:ball}
The ballot numbers $B_{n,r}$ are defined by
$B_{n,r}=\frac{r+1}{2n+r+1}\binom{2n+r+1}{n}$. They count the number
of paths $p$ from $(0,0)$ to $(2n+r,r)$ with allowed $U$-step and
$D$-step and each path $p$ can be decomposed as $p_1Up_2\cdots
Up_{r+1}$ with $p_i$, $1\le i\le r+1$, a Dyck path. Let
$\frac{M_{n,r}(q)}{(1-q)^n}$ be the generating function for the path
$p_1p_2\cdots p_{r+1}$ that has total length $2n$ where each peak of
$p_1p_2\cdots p_{r+1}$ has weight $q$ and each $U$-step of
$p_1p_2\cdots p_{r+1}$ has weight $\frac{1}{1-q}$. Then we can
generalize (\ref{E:koshyLa}) to an equation for $M_{n,r}(q)$,
i.e.,
\begin{eqnarray*}
\frac{M_{n,r}(q)}{(1-q)^{n}}=\frac{q(r+1)}{1-q}+\sum_{m=1}^{n-1}(-1)^{m-1}
\sum_{k\ge m}\frac{q^mM_{n-k,r}(q)}{(1-q)^{n-k+m}}
\binom{k-1}{m-1}\binom{n-k+1+r}{m}.
\end{eqnarray*}
The proof follows similarly to that for (\ref{E:koshyLa}) and is
omitted here. Next we will prove (\ref{E:cnq1in1}) by involution and
the inclusion-exclusion method on the partitions and answer Andrews'
questions on the property of $T_r(n,q)$ given in (\ref{E:cnq1in}).
\section{Properties of $T_r(n,q)$}\label{S:Tr}
We say a polynomial $f(x)=a_nx^n+\cdots+a_0$ is reciprocal if
$f(x)=x^nf(\frac{1}{x})$, i.e., $a_r=a_{n-r}$. A sequence
$a_0,a_1,\ldots,a_n$ of real numbers is said to be unimodal if for
some $0\le j\le n$ we have $a_0\le\cdots\le a_{j-1}\le a_j\ge
a_{j+1}\ge \cdots \ge a_n$. We say a polynomial
$f(x)=a_nx^n+\cdots+a_1x+a_0$ is unimodal if the sequence
$a_0,a_1,\ldots, a_n$ is unimodal. If $f(x)$ and $g(x)$ are unimodal
and reciprocal polynomials with nonnegative coefficients, then
$f(x)g(x)$ is also unimodal and reciprocal. Here we say $f(x)$ is a
positive polynomial if all the coefficients of $f(x)$ are
nonnegative. We say the polynomial $f(x)$ has nonnegative
coefficients (resp. nonpositive coefficients) up to $x^r$ if for any
$0\le i\le r$, the coefficient of $x^i$ in the polynomial $f(x)$ is
nonnegative (resp. nonpositive). $\omega\in\mathbb{C}$ is a
primitive $k$-th root of unity if and only if $\omega^k=1$ and for
any $d\in\mathbb{Z}^+$ and $d<k$, $\omega^d\ne 1$. The $k$-th
cyclotomic polynomial $\Phi_k(x)\in\mathbb{Z}[x]$ is the polynomial
whose roots are the primitive $k$-th roots of unity. Let $\omega$ be
any primitive $k$-th root of unity, and the polynomial $f(x)$
satisfies $f(\omega)=0$. Then $\Phi_k(x)\mid f(x)$. The $q$-Lucas
theorem is the following:
\begin{proposition}[$q$-Lucas theorem]
Let $m,k,d$ be positive integers, and write $m=ad+b$ and $k=rd+s$,
where $0\le b,s\le d-1$. Let $\omega$ be any primitive $d$-th root
of unity. Then
\begin{align*}
{m\brack k}_{\omega}=\binom{a}{r}{b\brack s}_{\omega}.
\end{align*}
\end{proposition}
\begin{theorem}\label{T:1}
$T_r(n,q)$ is a polynomial in $q$. If $n$ is even,
$T_r(n,q)$ is a positive polynomial. If $n$ is odd, $(1+q)T_r(n,q)$
is a positive polynomial. In case $n=2r-1$, $T_{r}(2r-1,-q)$ is a
positive polynomial.
\end{theorem}
\begin{proof}
We can simplify $T_r(n,q)$ defined in (\ref{E:cnq1in}) into
\begin{align}\label{E:genT}
T_r(n,q)&=q^{r^2-r}\frac{1}{[n]_q}{n\brack r}_{q^2}{2n-2r\brack
n-1}_q\\
&=q^{r^2-r}{n-1\brack r}_{q^2}\frac{(1+q^n)}{[n-1]_q}{2n-2r-1\brack
n-2}_q\label{E:genT2} .
\end{align}
Consequently from (\ref{E:genT}) and (\ref{E:genT2}) we have
for $r\ge 1$,
\begin{align}\label{E:Tr21}
T_r(n,q)&=([n]_q-q[n-1]_q)T_r(n,q)\nonumber\\&=q^{r^2-r}{n\brack
r}_{q^2}{2n-2r\brack n-1}_q-q^{r^2-r}{n-1\brack
r}_{q^2}(q+q^{n+1}){2n-2r-1\brack n-2}_q
\end{align}
which implies $T_r(n,q)$ is a polynomial.

We will next show the polynomial $T_r(n,q)$ is positive for even $n$
and the polynomial $(1+q)T_r(n,q)$ is positive for odd $n$. Let
$d=\gcd(n,r)$, then
\begin{align*}
T_r(n,q)&=q^{r^2-r}{n\brack r}_{q^2}\frac{[d]_{q^2}}{[n]_{q^2}}
\frac{[n]_{q^2}}{[d]_{q^2}[2n-2r+1]_q} {2n-2r+1 \brack n}_q\\
&=q^{r^2-r}{n\brack r}_{q^2}\frac{[d]_{q^2}}{[n]_{q^2}}
\frac{(1+q^n)(1-q)}{1-q^{2d}}{2n-2r\brack n-1}_q.
\end{align*}
It has been proved by Brunetti {\it et al.} that ${n\brack
r}_{q^2}\frac{[d]_{q^2}}{[n]_{q^2}}$ is a positive polynomial based
on the fact that polynomial $[d]_{q}{n\brack r}_{q}$ is unimodal and
reciprocal \cite{Brunetti:01}. In the same way, let $m=(n,2r-1)$,
then $\frac{[m]_q}{[2n-2r+1]_q}{2n-2r+1 \brack n}_q$ is a positive
polynomial. It remains to prove
$\frac{(1+q^n)(1-q)}{1-q^{2d}}{2n-2r\brack n-1}_q$ is a positive
polynomial for even $n$ and
$\frac{(1+q^n)(1-q^2)}{1-q^{2d}}{2n-2r\brack n-1}_q$ is a positive
polynomial for odd $n$. Recall that $m=(n,2r-1)$, first we observe
$m$ must be odd and therefore
\begin{align*}
\frac{(1+q^n)(1-q)}{1-q^{2d}}{2n-2r\brack n-1}_q
=\frac{[m]_{q}[m]_{-q}}{[2n-2r+1]_q}{2n-2r+1\brack n}_q
\frac{[n]_{q^2}}{[m]_{q^2}[d]_{q^2}}.
\end{align*}
Here $\frac{[n]_q}{[m]_q[d]_q}$ is a polynomial since $\gcd(m,d)=1$
and $m\vert n$, $d\vert n$. Together with the fact that
$\frac{[m]_q}{[2n-2r+1]_q}{2n-2r+1 \brack n}_q$ is a polynomial, it
follows that $\frac{(1+q^n)(1-q)}{1-q^{2d}}{2n-2r\brack n-1}_q$ is a
polynomial. We next prove $\frac{(1+q^n)(1-q)}{1-q^{2d}}{2n-2r\brack
n-1}_q$ is a positive polynomial if $n$ is even. If $n$ is even,
then $\frac{1-q}{1-q^{2d}}{2n-2r\brack n-1}_q$ is a polynomial since
for any $x\mid(2d)$ and $x>1$, $x\mid (2n-2r)$ and $x\nmid (n-1)$.
If not, $x\mid(n-1)$, then $x\mid\gcd(n-1,2)=1$, contradicting the
assumption. Thus by using the $q$-Lucas theorem,
$\frac{1-q}{1-q^{2d}}{2n-2r\brack n-1}_q$ is a polynomial. To be
precise, for any $x\in\mathbb{Z}^+$ such that $x|(2d)$ and $x>1$,
suppose $2n-2r=c_1x$, $n-1=c_2x+t_1$ for some $t_1\ne 0$, let $\rho$
be any primitive $x$-th root of unity, then by applying the
$q$-Lucas theorem, we have ${2n-2r\brack n-1}_\rho={c_1x\brack
c_2x+t_1}_{\rho}=\binom{c_1}{c_2}{0\brack t_1}_{\rho}=0$. This shows
for any $x\mid (2d)$ and $x>1$, $\Phi_{x}(q)|{2n-2r\brack n-1}_q$.
Since $[2d]_q=\prod_{\substack{x|(2d)\\x>1}}\Phi_{x}(q)$ and any two
cyclotomic polynomials are relatively prime, we can conclude that
$\prod_{\substack{x|(2d)\\x>1}}\Phi_{x}(q)|{2n-2r\brack n-1}_q$ and
therefore $\frac{1-q}{1-q^{2d}}{2n-2r\brack
n-1}_q=\frac{1}{[2d]_q}{2n-2r\brack n-1}_q$ is a polynomial. In view
of the unimodality of polynomial ${2n-2r\brack n-1}_q$,
$(1-q){2n-2r\brack n-1}_q$ has nonnegative coefficients up to
$q^{\lfloor\frac{(n-1)(n-2r+1)}{2}\rfloor}$. Therefore we expand
$\frac{1}{1-q^{2d}}$ as a power series at $q=0$ and obtain
\begin{align*}
\frac{(1-q)}{1-q^{2d}}{2n-2r\brack n-1}_q&=(1-q){2n-2r\brack n-1}_q
+q^{2d}(1-q){2n-2r\brack n-1}_q+\cdots
\end{align*}
which implies that the polynomial
$\frac{(1-q)}{1-q^{2d}}{2n-2r\brack n-1}_q$ has nonnegative
coefficients up to $q^{\lfloor\frac{(n-1)(n-2r+1)}{2}\rfloor}$.
Together with the reciprocity of $\frac{(1-q)}{1-q^{2d}}{2n-2r\brack
n-1}_q$, we conclude that $\frac{(1-q)}{1-q^{2d}}{2n-2r\brack
n-1}_q$ has nonnegative coefficients and therefore $T_r(n,q)$ is a
positive polynomial if $n$ is even.

In the same way, we can prove
that $\frac{(1-q^2)}{1-q^{2d}}{2n-2r\brack n-1}_q$ is a positive
polynomial if $n$ is odd. It implies $(1+q)T_r(n,q)$ is a positive polynomial
if $n$ is odd.

By setting $q\rightarrow-q$ and $n\rightarrow2r-1$, we
can obtain
\begin{align*}
T_{r}(2r-1,-q)=q^{r^2-r}{2r-1 \brack
r}_{q^2}\frac{1+q}{1+q^{2r-1}}=q^{r^2-r}[2r-1]_q\,C_{r-1}(q^2),
\end{align*}
which indicates $T_{r}(2r-1,-q)$ has nonnegative coefficients.
\end{proof}
\section{Proof of (\ref{E:cnq1in1})}\label{S:prove2}
We transform $T_r(n,q)$ from (\ref{E:Tr21}) into
\begin{align}\label{E:Tr22}
T_r(n,q)&=q^{r^2-r}{n-1\brack r-1}_{q^2}{2n-2r+1\brack
n}_q-q^{r^2-r}{n-1\brack r}_{q^2}q{2n-2r-1\brack n-2}_q\\
&+q^{r^2+r}{n-1\brack r}_{q^2}{2n-2r-1\brack n}_q \mbox{ valid for }
n\ge 2r+1\nonumber,\mbox{ and}\\
\nonumber T_r(n,q)&=q^{r^2-r}{n-1\brack r-1}_{q^2}{2n-2r+1\brack
n}_q-q^{r^2-r}{n-1\brack r}_{q^2}q{2n-2r-1\brack n-2}_q\\
&\mbox{valid for } 2r-1\le n\le 2r\nonumber.
\end{align}
For simplicity, we set $a_{n}^{(r)}=q^{r^2+r}{n-1\brack
r}_{q^2}{2n-2r-1\brack n}_q$. Then we can write
$T_r(n,q)=a_{n}^{(r-1)}+a_{n}^{(r)}-q^{r^2-r}{n-1\brack
r}_{q^2}q{2n-2r-1\brack n-2}_q$ for $n\ge 2r+1$ and
$T_r(n,q)=a_{n}^{(r-1)}-q^{r^2-r}{n-1\brack r}_{q^2}q{2n-2r-1\brack
n-2}_q$ for $2r-1\le n\le 2r$. First we observe
\begin{align*}
\quad &\sum_{r=1}^{\lfloor\frac{n+1}{2}\rfloor-1}(-1)^{r-1}
(a_{n}^{(r-1)}+a_{n}^{(r)})
+(-1)^{\lfloor\frac{n+1}{2}\rfloor-1}a_n^{(\lfloor\frac{n+1}{2}\rfloor-1)}\\
&=a_{n}^{(0)}={2n-1\brack n}_q=C_n(q)+q{2n-1\brack n-2}_q
\end{align*}
and it remains to prove
\begin{align}\label{E:invT}
{2n-1\brack
n-2}_q=\sum_{r=1}^{\lfloor\frac{n+1}{2}\rfloor}(-1)^{r-1}q^{r^2-r}{n-1\brack
r}_{q^2}{2n-2r-1\brack n-2}_q.
\end{align}
We interpret (\ref{E:invT}) in terms of partitions. A
partition $\lambda$ is defined as a finite sequence of nonnegative
integers $(\lambda_1,\lambda_2,\cdots,\lambda_m)$ in the weakly
decreasing order $\lambda_1\ge \lambda_2\ge \cdots\ge\lambda_m$.
Each $\lambda_i\ne 0$ is called a part of $\lambda$. The number and
the sum of parts of $\lambda$ are denoted by $\ell(\lambda)$ and
$\vert\lambda\vert$, respectively. The partition-theoretic
interpretation of the $q$-binomial ${n+k\brack k}_q$ is ${n+k\brack
k}_q=\sum_{\substack{\ell(\lambda)\le k\\\lambda_1\le
n}}q^{\vert\lambda\vert}$ as given in \cite{Andrews:par}. Therefore
we have
\begin{align}\label{E:partheo}
q^{n-2r+1}{2n-2r-1\brack n-2}_q&=\sum_{\substack{\ell(\nu)=
n-2r+1\\\nu_1\le n-1}}q^{\vert\nu\vert}, \quad q^{r^2+r}{n-1\brack
r}_{q^2}=\sum_{\substack{\mu_1>\mu_2>\cdots\mu_r\ge 1\\\mu_i\le
n-i}}q^{2\vert\mu\vert}.
\end{align}
Given two partitions $\mu$ and $\nu$, say $\mu=(\mu_1,\ldots,\mu_m)$
and $\nu=(\nu_1,\ldots,\nu_n)$, let $\mu^2\cup\nu$ be the partition
whose parts are $\mu_1,\mu_1,\ldots,\mu_m,\mu_m,\nu_1,\ldots,\nu_n$
in the decreasing order. Let $\mu\cup\nu$ be the partition whose
parts are $\mu_1,\ldots,\mu_m,\nu_1,\ldots,\nu_n$ in the decreasing
order, let $\nu\backslash (\mu_1,\ldots,\mu_m)$ be the partition
obtained from $\nu$ by removing the parts equal to
$\mu_1,\ldots,\mu_m$.

For any pair $(\mu,\nu)$ such that $n-1\ge\mu_1>\cdots>\mu_r\ge 1$
where $\mu_i\le n-i$ and $n-1\ge\nu_1\ge \cdots\ge \nu_{n-2r+1}\ge
1$, let $x$ be the smallest part in the partition $\mu^2\cup\nu$
with repetition, we shall construct a new pair $(\mu',\nu')$ as
follows:
\begin{enumerate}
\item{If $\mu_r=x$, then we choose
$\mu'=(\mu_1,\cdots,\mu_{r-1})$ and $\nu'=\nu\cup (x,x)$.}
\item{Otherwise, $\nu_{n-2r}=\nu_{n-2r+1}=x$, we choose
$\mu'=(\mu_1,\cdots,\mu_{r},x)$ and $\nu'=\nu\backslash(x,x)$.}
\end{enumerate}
Consequently
$2\vert\mu\vert+\vert\nu\vert=2\vert\mu'\vert+\vert\nu'\vert$ but
the lengths of $\mu$ and $\mu'$ differ by $1$. Indeed the map
$f:(\mu,\nu)\mapsto(\mu',\nu')$ is an involution, since for any pair
$(\mu',\nu')$, we can define the inverse map
$g:(\mu',\nu')\mapsto(\mu,\nu)$ as follows. Let $y$ be the smallest
part in the partition $\mu'^2\cup\nu'$ with repetition. If
$y\not\in\mu'$, then we choose $\mu=\mu'\cup (y)$ and
$\nu=\nu'\backslash (y,y)$. Otherwise we choose
$\mu=\mu'\backslash(y)$ and $\nu=\nu'\cup(y,y)$. Clearly, $g$ is the
inverse of $f$. This involution leads to (\ref{E:invT}).

In exactly the same way as for the proof of (\ref{E:Koshy1}) and
(\ref{E:koshyLa}), we can prove (\ref{E:invT}) by counting the
partitions that have exactly $k$ different parts with repetition
labeled by $s$. Let $g_{n,k}(q^{n+1}{2n-1\brack n-2}_q)$ count the
partitions $\lambda$ with $\lambda_1\le n-1$ and $\ell(\lambda)=n+1$
such that there are exactly $k$ different parts with repetition.
Then the number of partitions $\lambda$ with $\lambda_1\le n-1$ and
$\ell(\lambda)=n+1$, such that there are exactly $r$ different parts
with repetition labeled by $s$, is
\begin{align}\label{E:iepar}
\sum_{k\ge r}\binom{k}{r}g_{n,k}(q^{n+1}{2n-1\brack
n-2}_q)&=q^{r^2+r}{n-1\brack r}_{q^2}q^{n-2r+1}{2n-2r-1\brack
n-2}_q.
\end{align}
Again, by employing the generating function of
$g_{n,k}(q^{n+1}{2n-1\brack n-2}_q)$, we have
\begin{align*}
G_n(x,q)&=\sum_{k\ge 0}g_{n,k}(q^{n+1}{2n-1\brack n-2}_q)x^k=
\sum_{r\ge 0}\frac{G_n^{(r)}(1,q)}{r!}(x-1)^r\\
&=\sum_{r\ge 0}q^{r^2+r}{n-1\brack r}_{q^2}q^{n-2r+1}{2n-2r-1\brack
n-2}_q(x-1)^r.
\end{align*}
By setting $x=0$ on both sides, we get (\ref{E:invT}) and therefore
(\ref{E:cnq1in1}) follows. From the proof we see that the expression
of $T_r(n,q)$ given in (\ref{E:Tr22}) implies that the sieving
process works on the partitions counted by ${2n-1\brack n-2}_q$.

\section{Generalize (\ref{E:cnq1in1}) to $q$-Ballot numbers}\label{S:qball}
The $q$-Ballot numbers $B_j(n,q)$ are defined by
\begin{align*}
B_j(n,q)=\frac{[j]_q}{[2n+j]_q}{2n+j\brack
n}_q=\frac{[j]_q}{[n]_q}{2n+j-1\brack n-1}_q.
\end{align*}
They count the major index of lattice paths from $(0,0)$ to
$(2n+j-1,-j+1)$ with allowed $U$-step and $D$-step that never go
below $y=-j+1$, see \cite{Kratten:89}. In particular,
$B_1(n,q)=C_{n}(q)$ and $B_r(n,1)=B_{n,r-1}$. Furthermore,
\begin{align*}
B_{j}(n,q)&={2n+j-2\brack n}_q-q^{j}{2n+j-2\brack n-2}_q
\end{align*}
also counts the partitions $\lambda$ with $\lambda_1\le n+j-2$ and
$\ell(\lambda)\le n$ whose successive ranks are all $<j-1$, see
\cite{Andrews:93}. We generalize (\ref{E:cnq1in1}) to an
equation for $q$-Ballot numbers as follows.
\begin{theorem}
The $q$-Ballot numbers $B_j(n,q)$ satisfy
\begin{align}\label{E:qBall}
B_j(n,q)&=\sum_{r=1}^{n}(-1)^{r-1}T_r^{(j)}(n,q),\\
\label{E:qBall1} T_r^{(j)}(n,q)&=q^{r^2-r}{n\brack
r}_{q^2}{2n+j-1-2r\brack n-1}_q\frac{[j]_q}{[n]_q},
\end{align}
where $T_r^{(j)}(n,q)$ is a polynomial for $n\ge 2r-j$. In
particular $T_r^{(j)}(2r-j,-q)$ is a positive polynomial if $j\le
r$.
\end{theorem}
\begin{proof}
First we shall show $T_r^{(j)}(n,q)$ is a polynomial for $n\ge
2r-j$. By setting $d=\gcd(n,r)$ and $m=\gcd(n,2r-j)$ we express
\begin{align*}
T_r^{(j)}(n,q)&=q^{r^2-r}{n\brack
r}_{q^2}\frac{[d]_{q^2}}{[n]_{q^2}}{2n+j-1-2r\brack
n-1}_q\frac{[n]_{q^2}[j]_q}{[d]_{q^2}[n]_q}\\
&=q^{r^2-r}{n\brack r}_{q^2}\frac{[d]_{q^2}}{[n]_{q^2}}
\frac{[n]_{q^2}[j]_q}{[d]_{q^2}[m]_q}\frac{[m]_q}{[2n+j-2r]_q}
{2n+j-2r\brack n}_q
\end{align*}
where ${n\brack r}_{q^2}\frac{[d]_{q^2}}{[n]_{q^2}}$ and
$\frac{[m]_q}{[2n+j-2r]_q} {2n+j-2r\brack n}_q$ are positive
polynomials as proved in Theorem~\ref{T:1} and it remains to show
$\frac{[n]_{q^2}[j]_q}{[d]_{q^2}[m]_q}$ is a polynomial. First we
can simplify
\begin{align*}
\frac{[n]_{q^2}[j]_q}{[d]_{q^2}[m]_q}=
\frac{[2n]_q[j]_q}{[2d]_q[m]_q}.
\end{align*}
For any $m_1\mid m$ and $m_1>1$, if $m_1\mid j$, then
$\Phi_{m_1}(q)\mid[j]_q$ and therefore
$\Phi_{m_1}(q)\mid\frac{[2n]_q}{[2d]_q}[j]_q$. Otherwise, $m_1\nmid
j$ and therefore $m_1\nmid d$. In fact, $m_1\nmid (2d)$. If not,
then $m_1\mid (2r)$. Since $m_1\mid m$, hence $m_1\mid (2r-j)$. That
leads to $m_1\mid j$, contradicting the assumption. Let $\omega$ be
any $m_1$-the primitive root of unity, i.e., $\omega^{m_1}=1$. Then
we have $[2n]_{\omega}=0$ and $[2d]_{\omega}\ne 0$ which implies
$\Phi_{m_1}(q)\mid \frac{[2n]_q}{[2d]_q}$. It follows that
$\Phi_{m_1}(q)\mid \frac{[2n]_q}{[2d]_q}[j]_q$ for any $m_1\vert m$
and $m_1>1$. Since any two cyclotomic polynomials are relatively
prime, we have that $\frac{[2n]_q[j]_q}{[2d]_q[m]_q}$ is a
polynomial. Therefore
\begin{align}\label{E:unimo}
\frac{[n]_{q^2}[j]_q}{[d]_{q^2}[m]_q}
\frac{[m]_q}{[2n+j-2r]_q}{2n+j-2r\brack n}_q
=\frac{(1+q^n)(1-q)}{(1-q^{2d})}[j]_q{2n+j-2r-1\brack n-1}_q
\end{align}
and $T_r^{(j)}(n,q)$ are polynomials.

In particular, $T_r^{(j)}(2r-j,q)$ is a polynomial. We next prove
$T_r^{(j)}(2r-j,-q)$
is a positive polynomial if $j\le r$. If $j$ is even, then
\begin{align*}
T_r^{(j)}(2r-j,-q)&=q^{r^2-r}{2r-j\brack
r}_{q^2}\frac{1-q^j}{1-q^{2r-j}}
\end{align*}
holds. Here $q^{r^2-r}{2r-j\brack r}_{q^2}(1-q^j)$ has nonnegative
coefficients up to $q^{\lfloor\frac{3r^2-r-2rj}{2}\rfloor}$ since
$q^{\frac{r^2-r}{2}}{2r-j\brack r}_q$ is unimodal and $j$ is even.
In combination of the expansion of $\frac{1}{1-q^{2r-j}}$ at $q=0$
where $j\le r$, we conclude that $T_r^{(j)}(2r-j,-q)$ is a positive
polynomial for even $j$. However, if $j$ is odd, we cannot prove the
claim in the same way as before. If $j$ is odd, we have
\begin{align}\label{E:jod}
T_r^{(j)}(2r-j,-q)&=q^{r^2-r}{2r-j\brack
r}_{q^2}\frac{1+q^j}{1+q^{2r-j}}.
\end{align}
Since $q^{\frac{r^2-r}{2}}{2r-j\brack r}_{q}$ is a positive
polynomial and $j$ is odd, we find that $(1+q^j)q^{r^2-r}{2r-j\brack
r}_{q^2}$ is a positive polynomial. But in view of the expansion of
$\frac{1}{1+q^{2r-j}}$ at $q=0$, we get
\begin{align*}
T_r^{(j)}(2r-j,-q)=(1+q^j)q^{r^2-r}{2r-j\brack
r}_{q^2}-q^{2r-j}(1+q^j)q^{r^2-r}{2r-j\brack r}_{q^2}+\cdots
\end{align*}
where the polynomial $-q^{2r-j}(1+q^j)q^{r^2-r}{2r-j\brack r}_{q^2}$
could make the first half of the coefficients of
$(1+q^j)q^{r^2-r}{2r-j\brack r}_{q^2}$ become negative. Therefore,
we choose to prove $T_r^{(j)}(2r-j,-q)$ in (\ref{E:jod}) to be a
positive polynomial by considering
\begin{align*}
qT_r^{(j)}(2r-j,q)=q^{r^2-r+1}{2r-j\brack
r}_{q^2}\frac{1-q^j}{1-q^{2r-j}}.
\end{align*}
The claim that polynomial $T_r^{(j)}(2r-j,-q)$ has nonnegative
coefficients is equivalent to the claim that polynomial
$qT_r^{(j)}(2r-j,q)$ has nonnegative coefficients for the odd powers
of $q$, and nonpositive coefficients for the even powers of $q$. Now
for $qT_r^{(j)}(2r-j,q)$ the unimodality of
$q^{\frac{r^2-r}{2}}{2r-j\brack r}_q$ implies that
$q^{r^2-r+1}{2r-j\brack r}_{q^2}(1-q^j)$ has nonnegative
coefficients for the odd powers of $q$, up to
$q^{\lfloor\frac{3r^2-r-2rj+1}{2}\rfloor}$. By expanding
$\frac{1}{1-q^{2r-j}}$ as a power series at $q=0$, we can conclude
that the coefficients of the odd powers of $q$, up to
$q^{\lfloor\frac{3r^2-r-2rj+1}{2}\rfloor}$ in the polynomial
$qT_r^{(j)}(2r-j,q)$ are nonnegative. Furthermore, we observe that
the maximal degree $d_{\max}$ of $q$ in the polynomial
$qT_r^{(j)}(2r-j,q)$ is $3r^2-3r-2rj+2j+1$ and the minimal degree
$d_{\min}$ of $q$ in the polynomial $qT_r^{(j)}(2r-j,q)$ is
$r^2-r+1$. Let $a_i$ be the coefficient of $q^i$ in the polynomial
$qT_r^{(j)}(2r-j,q)$. Then from the reciprocity of the polynomial
$qT_r^{(j)}(2r-j,q)$ we get $a_i=a_j$ if
$j+i=d_{\max}+d_{\min}=4r^2-2rj-4r+2j+2$. That implies for $i$ odd
and $a_i\ge 0$, $j$ is odd and $a_j\ge 0$. Therefore we can conclude
that the coefficients of all the odd powers of $q$ in the polynomial
$qT_r^{(j)}(2r-j,q)$ are nonnegative. That is equivalent to say that
the coefficients of all the even powers of $q$ in the polynomial
$T_r^{(j)}(2r-j,q)$ are nonnegative. By expanding
$\frac{1-q^j}{1-q^{2r-j}}$ at $q=0$, we have that the even powers of
$q$ in the polynomial $qT_r^{(j)}(2r-j,q)$ come from
\begin{align*}
-q^j(1+q^{4r-2j}+q^{8r-4j}+\cdots)(1-q^{2r-2j})q^{r^2-r+1}{2r-j\brack
r}_{q^2}
\end{align*}
where $-q^{r^2-r+1+j}{2r-j\brack r}_{q^2}(1-q^{2r-2j})$ is a
polynomial that has nonpositive coefficients for the even powers of
$q$, up to $q^{\lfloor\frac{3r^2-2rj-r+j+1}{2}\rfloor}$. Again from
the reciprocity of polynomial $qT_r^{(j)}(2r-j,q)$, we conclude that
the coefficients of all the even powers of $q$ in the polynomial
$qT_r^{(j)}(2r-j,q)$ are nonpositive. That is to say, the
coefficients of all the odd powers of $q$ in the polynomial
$T_r^{(j)}(2r-j,-q)$ are nonnegative. Now it remains to prove
(\ref{E:qBall}). By similar techniques as used to prove
(\ref{E:cnq1in1}), we consider the partitions counted by
$q^{n+j}{2n+j-1\brack n-1}_q$ and let $g_{n,k}(q^{n+j}{2n+j-1\brack
n-1}_q)$ count the partitions $\lambda$ with $\lambda_1\le n$ and
$\ell(\lambda)=n+j$ such that there are exactly $k$ different parts
with repetition. By following the same techniques used in the proof
of (\ref{E:cnq1in1}) of Section~\ref{S:prove2}, see the proof of
(\ref{E:iepar}), we can derive the identity
\begin{align*}
q^{n+j}{2n+j-1\brack
n-1}_q&=\sum_{r=1}^{n}(-1)^{r-1}q^{r^2+r}{n\brack
r}_{q^2}q^{n+j-2r}{2n+j-2r-1\brack n-1}_q,
\end{align*}
which is equivalent to (\ref{E:qBall}).
\end{proof}
\section{Conjecture}
Here we conjecture that if $n$ is odd, then for $m\ge n\ge 1$ the
polynomial $(1+q^n){m\brack n-1}_q$ is unimodal. If $n$ is even,
then for any even $j\ne 0$ and $m\ge n\ge 1$, the polynomial
$(1+q^n)[j]_q{m\brack n-1}_q$ is unimodal. This, in combination of
(\ref{E:unimo}), implies that $T_r(n,q)$ is a positive polynomial
for odd $n$ and $n\ge 2r+1$, and $T_r^{(j)}(n,q)$ given in
(\ref{E:qBall1}) is a positive polynomial for $n\ge 2r-j+1$.
\section*{Acknowledgement}
The authors would like to thank the anonymous reviewers for their
helpful comments that greatly helped to improve the final version of
this manuscript.


\end{document}